\documentclass[a4paper,11pt]{amsart}

\usepackage[top=30truemm,bottom=25truemm,left=35truemm,right=35truemm]{geometry}

\usepackage{amsmath}
\usepackage{amssymb}
\usepackage{amsthm}
\usepackage{graphics}
\usepackage[abbrev,alphabetic]{amsrefs}
\usepackage{amscd}
\usepackage[dvipdfmx]{graphicx}
\usepackage{color}
\usepackage{ulem}
\usepackage{enumitem}

\usepackage[all,cmtip]{xy}
\usepackage{xypic}

\title[A non-klt counterexample to Shokurov's index conjecture]
{A non-klt counterexample to Shokurov's index conjecture}

\author{Yusuke Nakamura}
\address{Graduate School of Mathematics, Nagoya University, Furo-cho, Chikusa-ku, Nagoya, 464-8602, Japan.}
\email{y.nakamura@math.nagoya-u.ac.jp}
\urladdr{https://sites.google.com/site/ynakamuraagmath/}

\author{Kohsuke Shibata}
\address{School of Engineering, Tokyo Denki University, Adachi-ku, Tokyo 120-8551, Japan.}
\email{shibata.kohsuke@mail.dendai.ac.jp}



\newtheorem{thm}{Theorem}[section]

\newtheorem{prop}[thm]{Proposition}

\newtheorem{claim}[thm]{Claim}

\theoremstyle{definition}
\newtheorem{defi}[thm]{Definition}
\newtheorem{eg}[thm]{Example}
\newtheorem{conj}[thm]{Conjecture}

\theoremstyle{remark}
\newtheorem{rmk}[thm]{Remark}

\newtheorem*{ackn}{Acknowledgements}

\begin{document}

\maketitle

\begin{abstract}
We give a counterexample to Shokurov's index conjecture for minimal log discrepancies in the non-klt setting.
\end{abstract}

\section{Introduction}

Throughout this paper, we work over an algebraically closed field $k$ of characteristic zero. 

Minimal log discrepancies are fundamental invariants of singularities in birational geometry. 
They measure the complexity of singularities and play an important role in the minimal model program, particularly in the context of the termination of flips (cf.\ \cite{Sho04}). 
The following conjecture, proposed by Shokurov, predicts that the minimal log discrepancy bounds the Cartier index of the canonical divisor. 

\begin{conj}[Shokurov, cf.\ \cite{Kaw15}*{Question 5.2}]\label{conj:index}
For any $n \in \mathbb{Z}_{> 0}$ and $a \in \mathbb{R}_{\ge 0}$, there exists a positive integer $r(n,a)$ satisfying the following condition: 
\begin{itemize}
\item If an $n$-dimensional $\mathbb{Q}$-Gorenstein variety $X$ and a closed point $p \in X$ satisfy 
$\operatorname{mld}_p(X) = a$, 
then the Cartier index of $K_X$ at $p$ is at most $r(n,a)$. 
\end{itemize}
\end{conj}
\noindent
Here, the \textit{Cartier index} of $K_X$ at $p$ is defined as the smallest positive integer $r$ such that $rK_X$ is Cartier in an open neighborhood of $p$. 

Conjecture~\ref{conj:index} has been verified in several cases. 
Historically, the specific case of $(n,a)=(2,0)$ follows from the classification of surface singularities (see \cite{Sho93}*{Corollary 5.10}). 
For general surfaces, the conjecture was proved by Chen and Han~\cite{CH21}, who established the result in the setting of pairs. 
In dimension three, the case $a=0$ was established by Ishii~\cite{Ish00} and Fujino~\cite{Fuj01}. 
Furthermore, Fujino proved the case $a=0$ in arbitrary dimensions, subject to a boundedness conjecture for birational automorphism groups \cite{Fuj01}*{Theorem 0.2}. 
For three-dimensional terminal singularities, the conjecture follows from Kawamata's classification~\cite{Kaw92}, while for three-dimensional canonical singularities, it was solved by Kawakita~\cite{Kaw15}. This was generalized to the setting of three-dimensional terminal pairs by Han, Liu, and Luo~\cite{HLL25}*{Theorem 1.5}. 
Finally, the conjecture has been confirmed for specific classes of singularities in arbitrary dimensions: for toric singularities by Ambro~\cite{Amb09} (cf.\ \cite{Amb16}*{Section 5.4}), and for quotient singularities by the authors~\cite{NS25}, following a partial result by Moraga~\cite{Mor23}.

In contrast to these positive results, we show that Shokurov's index conjecture does not hold in general. 
More precisely, we construct three-dimensional normal $\mathbb{Q}$-Gorenstein singularities with minimal log discrepancy exactly $1$ and arbitrarily large Cartier index.

\begin{thm}[$=$ Theorem \ref{thm:counterexample}]\label{thm:main}
For every integer $m \ge 1$, there exist a three-dimensional normal $\mathbb{Q}$-Gorenstein variety $X_m$ and a closed point $0 \in X_m$ such that $\operatorname{mld}_0(X_m)=1$ and the Cartier index of $K_{X_m}$ at $0$ is exactly $m$.
In particular, Conjecture~\ref{conj:index} fails for $(n,a) = (3,1)$.
\end{thm}

We briefly describe the construction (see Example \ref{eg:X_m} for more details). 
Let $Y$ be an affine cone over an elliptic curve, polarized by an ample line bundle of degree $m$. 
A torsion translation of order $m$ on the elliptic curve naturally induces an action of the cyclic group $G \simeq \mathbb{Z}/m\mathbb{Z}$ on $Y$. 
We then extend this action to the product $Y \times \mathbb{A}^1$ by letting the generator act on the coordinate $t$ of $\mathbb{A}^1$ as $t \mapsto \xi t$, where $\xi$ is a primitive $m$-th root of unity. 
We define our counterexample as the quotient variety $X_m := (Y \times \mathbb{A}^1)/G$. 
By introducing an $\mathbb{A}^1$-factor and a scaling action on it, this construction forces the Cartier index of $K_{X_m}$ to be exactly $m$.

The mechanism that preserves the minimal log discrepancy in this construction relies on the concept of \textit{virtually free} actions. 
The notion of virtually free actions was recently introduced by the authors in \cite{NS4} to construct counterexamples to the precise inversion of adjunction (PIA) conjecture. 
As explained in \cite{NS4}*{Remark 3.2}, a virtually free action is an intermediate concept between a free action and an action that is free in codimension one. 
A remarkable feature of virtually free actions is that they preserve minimal log discrepancies under quotients, exactly as free actions do (see Proposition \ref{prop:eq}). 
In our construction, the $G$-action on $Y \times \mathbb{A}^1$ is virtually free. 
This property guarantees that $\operatorname{mld}_0(X_m) = \operatorname{mld}_{(0,0)}(Y \times \mathbb{A}^1) = 1$, allowing us to control the minimal log discrepancy independently of the Cartier index.

We conclude with a remark on the singularity class. 
Since our counterexample $0 \in X_m$ is log canonical but not klt, it suggests a natural revision for Conjecture~\ref{conj:index}: one should either assume that $X$ is klt, or, in the log canonical case, restrict to $a=0$.
Indeed, it is shown in \cite{NS4}*{Theorem 3.10} (cf.\ \cite{NS5}*{Theorem 6.7}) that there exists no non-free virtually free action on klt singularities. 
Therefore, one cannot construct a counterexample in the klt setting using the same idea based on virtually free actions.

\begin{ackn} 
The first author is partially supported by Inamori Foundation and by JSPS KAKENHI No.\ 22K13888. 
The second author is partially supported by JSPS KAKENHI No.\ 23K12958.
\end{ackn}

\section{Preliminaries}

We recall the definition of the minimal log discrepancy for normal $\mathbb{Q}$-Gorenstein varieties.

\begin{defi}\label{defi:mld}
Let $X$ be a normal variety and let $p \in X$ be a closed point. 
Suppose that $K_X$ is $\mathbb{Q}$-Cartier in a neighborhood of $p$.
Let $f\colon X' \to X$ be a proper birational morphism from a normal variety $X'$, and let $E$ be a prime divisor on $X'$ with center $c_X(E) = \{p\}$. 
The \textit{log discrepancy} of $E$ with respect to $X$ is defined by
\[
a_E(X) := 1 + \operatorname{ord}_E(K_{X'} - f^* K_X).
\]

The \textit{minimal log discrepancy} of $X$ at $p$ is defined by
\[
\operatorname{mld}_p(X) := \inf_{c_X(E) = \{ p \}} a_E(X)
\]
if $\dim X \ge 2$, where the infimum is taken over all prime divisors $E$ over $X$ with center $c_X(E) = \{ p \}$. 
When $\dim X=1$, we use the standard convention that
$\operatorname{mld}_p(X)$ is this infimum if it is non-negative, and is $-\infty$ otherwise.
It is known that $\operatorname{mld}_p(X) \in \mathbb{Q}_{\ge 0} \cup \{ - \infty \}$ in general (cf.\ \cite{KM98}*{Corollary 2.31}). 
\end{defi}

Following \cite{NS4}, we define the concept of virtually free actions. 

\begin{defi}\label{defi:vf}
Suppose that a finite group $G$ acts on a normal variety $X$. 
We say that the $G$-action on $X$ is \textit{virtually free} if the following condition holds: 
for every $G$-equivariant proper birational morphism $Y \to X$ from a normal variety $Y$, the $G$-action on $Y$ is free in codimension one. 
\end{defi}

\begin{rmk}\label{rmk:vf}
By \cite{NS4}*{Proposition 3.8}, this condition is equivalent to saying that there exists a $G$-equivariant proper birational morphism $Y \to X$ from a normal variety $Y$ such that the $G$-action on $Y$ is free. 
\end{rmk}

The following proposition states an important property of virtually free actions: it asserts that the minimal log discrepancy is preserved under the quotient, precisely as if the action were free.

\begin{prop}[\cite{NS4}*{Proposition 3.5}]\label{prop:eq}
Let $X$ be a normal $\mathbb{Q}$-Gorenstein variety equipped with an action of a finite group $G$. 
Suppose that the $G$-action on $X$ is virtually free. 
Then, we have 
\[
\operatorname{mld}_{x}(X) = \operatorname{mld}_{x'}(X/G)
\]
for any closed point $x \in X$ and its image $x' \in X/G$. 
\end{prop}

\section{Construction of the counterexamples}

In this section, we construct the counterexamples in Theorem \ref{thm:main}. 
The idea is to consider the product of an affine cone over an elliptic curve and the affine line $\mathbb{A}^1$. 
By using a torsion translation on the elliptic curve, we construct a cyclic group action on this product that is virtually free. 
Since the minimal log discrepancy is preserved under virtually free quotients, taking a translation of an arbitrarily large order allows us to make the Cartier index arbitrarily large while keeping the minimal log discrepancy unchanged.

\begin{eg}[$X_m$]\label{eg:X_m}
Let $m$ be a positive integer, and let $C$ be an elliptic curve over $k$. 
Let $L$ be a line bundle on $C$ with $\deg L = m$.  
Since $L$ is ample, we can define the affine cone $Y$ over $C$ polarized by $L$ as 
\[
R := \bigoplus _{n \ge 0} \Gamma (C, L^{\otimes n}), \qquad 
Y := \operatorname{Spec} R. 
\]
Then $Y$ is a normal Gorenstein variety (cf.\ \cite{Kol13}*{(3.8) and Proposition 3.14(4)}). 
Let $0 \in Y$ denote the vertex of the cone, which is the closed point corresponding to the ideal $\bigoplus_{n > 0} \Gamma(C, L^{\otimes n}) \subset R$. 

Let $P \in C(k)$ be a closed point of order $m$. 
Consider the translation by $P$ on $C$:
\[
t_P \colon C \longrightarrow C, \qquad Q \longmapsto Q+P.
\]
Since $\deg L = m$ and $P$ is an $m$-torsion point, we have an isomorphism $t_P^* L \cong L$. 
For an arbitrary isomorphism $\alpha_0 \colon t_P^* L \xrightarrow{\sim} L$, the composition $\alpha_0\circ t_P^*\alpha_0\circ\cdots\circ(t_P^{m-1})^*\alpha_0$ acts on $L$ as multiplication by some scalar $c \in k^\times$. 
By rescaling it as $\alpha := c^{-1/m}\alpha_0$, we obtain an isomorphism $\alpha \colon t_P^* L \xrightarrow{\sim} L$ such that $\alpha\circ t_P^*\alpha\circ\cdots\circ(t_P^{m-1})^*\alpha$ is the identity on $L$. 
By fixing this isomorphism $\alpha$, the translation $t_P$ induces an automorphism $T$ on $Y$ of order exactly $m$.

Let $G = \mathbb{Z}/m\mathbb{Z}$ be the cyclic group. 
We define the action of $G$ on $Y$ via the generator $T$. 
Let $\xi \in k$ be a primitive $m$-th root of unity. 
We define the action of $G$ on $\mathbb{A}^1 = \operatorname{Spec} k[t]$ such that the generator acts by multiplication by $\xi$ (i.e., $t \mapsto \xi t$). 
We consider the diagonal action of $G$ on the product $Y \times \mathbb{A}^1$ and define the quotient variety
\[
X_m := (Y \times \mathbb{A}^1)/G.
\]
Let $0 \in X_m$ denote the image of the point $(0, 0) \in Y \times \mathbb{A}^1$ under the quotient morphism. 
\end{eg}

\begin{thm}\label{thm:counterexample}
Let $X_m$ be the three-dimensional normal variety and $0 \in X_m$ the closed point constructed in Example \ref{eg:X_m}. 
Then, the following hold:
\begin{enumerate}
\item $\operatorname{mld}_0(X_m) = 1$. 
\item The Cartier index of $K_{X_m}$ at $0$ is exactly $m$. 
\end{enumerate}
\end{thm}

\begin{proof}
We define 
\[
\widetilde{Y} := \underline{\operatorname{Spec}}_C \bigoplus_{n \ge 0} L^{\otimes n}.
\]
Let $\pi \colon \widetilde{Y} \to C$ be the structural morphism. 
The morphism $\pi$ is an $\mathbb{A}^1$-bundle, and hence $\widetilde{Y}$ is smooth.
There naturally exists a proper birational morphism $p \colon \widetilde{Y} \to Y$. 
The morphism $p$ is an isomorphism outside the vertex $0$, and the exceptional divisor $E := p^{-1}(0) \subset \widetilde{Y}$ is isomorphic to $C$. 
In fact, the restriction $\pi|_E \colon E \to C$ gives an isomorphism.
\[
\xymatrix{
  E \ar@{}[r]|-{\subset} \ar[dr]_{\pi|_E}^{\hspace{-2mm}\simeq} & \widetilde{Y} \ar[d]^{\pi} \ar[r]^{p} & Y \\
  & C \ar@{|->}[r] & 0 \ar@{}[u]|-{\rotatebox{90}{$\in$}}
}
\]
We also have 
\[
p^* K_{Y} = K_{\widetilde{Y}} + E, 
\]
and therefore $\operatorname{mld}_0 (Y) = 0$ (cf.\ \cite{Kol13}*{(3.8) and Proposition 3.14(4)}). 

The translation $t_P$ and $\alpha$ also induce an automorphism $\widetilde{T}$ on $\widetilde{Y}$. 
Via the isomorphism $\pi|_E \colon E \xrightarrow{\sim} C$, the restriction of $\widetilde{T}$ to $E$ naturally coincides with $t_P$. 
We define the actions of $G = \mathbb{Z}/m\mathbb{Z}$ on $C$ and $\widetilde{Y}$ via $t_P$ and $\widetilde{T}$, respectively. 
By construction, the morphisms $p \colon \widetilde{Y} \to Y$ and $\pi \colon \widetilde{Y} \to C$ are $G$-equivariant. 
Since $P$ has order $m$, the group $G$ acts freely on $C$. 
Because $\pi$ is $G$-equivariant, it immediately follows that the action of $G$ on $\widetilde{Y}$ is also free. 
On the other hand, for $m \ge 2$, the vertex $0 \in Y$ is the unique fixed point of the action of $G$ on $Y$ (while for $m=1$, the group $G$ is trivial). 
In either case, the $G$-action on $Y$ is virtually free by Remark \ref{rmk:vf}. 

It follows from \cite{NS4}*{Theorem 3.9} that the diagonal action of $G$ on the product $Y \times \mathbb{A}^1$ is also virtually free. 
Therefore, by applying Proposition \ref{prop:eq}, we obtain
\[
\operatorname{mld}_0(X_m) = \operatorname{mld}_{(0,0)} \bigl( Y \times \mathbb{A}^1 \bigr) = \operatorname{mld}_0(Y) + \operatorname{mld}_0(\mathbb{A}^1) = 0 + 1 = 1.
\]
This completes the proof of assertion (1).

To simplify the notation, we define
\[
\overline{X}_m := Y \times \mathbb{A}^1 \quad \text{and} \quad \overline{W}_m := \widetilde{Y} \times \mathbb{A}^1.
\]
The group $G$ acts diagonally on these products, where the action on the second factor $\mathbb{A}^1$ is given by $t \mapsto \xi t$. 
We define the quotient variety 
\[
W_m := \overline{W}_m / G = \bigl( \widetilde{Y} \times \mathbb{A}^1 \bigr) / G.
\]
Let $F := E \times \mathbb{A}^1 \subset \overline{W}_m$, and let $F' \subset W_m$ be its image under the quotient morphism $h \colon \overline{W}_m \to W_m$. 
These spaces and morphisms fit into the following commutative diagram:
\[
\xymatrix@C=3em{
  F := E \times \mathbb{A}^1 \ar@{}[r]|-{\subset} \ar[d] & \overline{W}_m := \widetilde{Y} \times \mathbb{A}^1 \ar[d]_{h} \ar[r]^-{q \,=\, p \times \operatorname{id}} & \overline{X}_m := Y \times \mathbb{A}^1 \ar[d]^{g} \\
  F' \ar@{}[r]|-{\subset} & W_m := \bigl( \widetilde{Y} \times \mathbb{A}^1 \bigr) / G \ar[r]^-{r} & X_m := (Y \times \mathbb{A}^1) / G
}
\]

Recall that the $G$-action on $\widetilde{Y}$ is free. 
Thus, the diagonal action on $\overline{W}_m = \widetilde{Y} \times \mathbb{A}^1$ is also free. 
This implies that $W_m$ is smooth, and the quotient morphism $h$ is \'{e}tale. 
In particular, we have $K_{\overline{W}_m} = h^* K_{W_m}$. 
Since $F$ is a smooth divisor invariant under the free $G$-action, its image $F'$ is also a smooth divisor on $W_m$, and $h^* F' = F$. 

Since $Y$ is Gorenstein, so is $\overline{X}_m$. 
Because $g \colon \overline{X}_m \to X_m$ is the quotient morphism by the finite group $G$ of order $m$, it follows that $m K_{X_m}$ is a Cartier divisor. 
This shows that the Cartier index of $K_{X_m}$ at $0$ divides $m$. 

To show that the index is exactly $m$, let $c$ denote the Cartier index of $K_{X_m}$ at $0$. 
Let $r \colon W_m \to X_m$ and $q \colon \overline{W}_m \to \overline{X}_m$ be the morphisms in the commutative diagram above. 
Since the $G$-action on $\overline{X}_m$ is free in codimension one, we have the equality $g^* K_{X_m} = K_{\overline{X}_m}$ of $\mathbb{Q}$-divisors. 
Since $q^* K_{\overline{X}_m} = K_{\overline{W}_m} + F$, we obtain
\[
h^* (r^* K_{X_m}) = q^* (g^* K_{X_m}) = q^* K_{\overline{X}_m} = K_{\overline{W}_m} + F = h^* (K_{W_m} + F'). 
\]
Since $h$ is a surjective finite \'{e}tale morphism, this implies the equality $r^* K_{X_m} = K_{W_m} + F'$ of $\mathbb{Q}$-divisors. 
Multiplying by $c$, we obtain the equality of Cartier divisors
\[
r^*(c K_{X_m}) = c(K_{W_m} + F'). 
\]

Let $F'_0 \subset F'$ be the image of $E \times \{0\}$ under the quotient morphism $F \to F'$. 
Note that $F'_0$ is naturally isomorphic to the quotient curve $E/G$. 
Since $r(F'_0) = \{0\}$, we have
\[
\mathcal{O}_{W_m}\bigl( c(K_{W_m} + F') \bigr)|_{F'_0} = \mathcal{O}_{W_m} \bigl( r^*(c K_{X_m}) \bigr)|_{F'_0} \cong \mathcal{O}_{F'_0}. 
\]
Since $(K_{W_m} + F')|_{F'} \sim K_{F'}$ by the adjunction formula, we have
\[
\mathcal{O}_{F'}(c K_{F'})|_{F'_0} \cong \mathcal{O}_{F'_0}. 
\]

Recall that the $G$-action on $F = E \times \mathbb{A}^1$ is diagonal, acting freely on $E$ and linearly on $\mathbb{A}^1$. 
Thus, the quotient $F' = ( E \times \mathbb{A}^1 ) /G$ is an $\mathbb{A}^1$-bundle over the quotient curve $E/G$, and $F'_0$ is embedded as its section. 
Since $F'$ is an $\mathbb{A}^1$-bundle over the smooth curve $E/G$, the pullback $\operatorname{Pic}(E/G) \to \operatorname{Pic}(F')$ is an isomorphism (cf.\ \cite{Ful}*{Theorem 3.3(a)}). 
Its inverse is given by the restriction to the section $F'_0$, so the restriction map $\operatorname{Pic}(F') \to \operatorname{Pic}(F'_0)$ is also an isomorphism. 
This yields 
\[
\mathcal{O}_{F'}(c K_{F'}) \cong \mathcal{O}_{F'}.
\]

\begin{claim}\label{claim}
The line bundle $\mathcal{O}_{F'}(K_{F'})$ is a torsion element of order exactly $m$ in $\operatorname{Pic}(F')$. 
\end{claim}

\begin{proof}[Proof of Claim]
By construction, $F'$ is the quotient of $F = E \times \mathbb{A}^1 \cong C \times \mathbb{A}^1$ by the free action of $G$. 
Since $C$ is an elliptic curve, its canonical bundle is trivial and generated by a nowhere-vanishing global regular differential form $\omega_C$. 
Similarly, the canonical bundle of $\mathbb{A}^1 = \operatorname{Spec} k[t]$ is generated by $dt$. 
Thus, the canonical bundle of $F$ is trivial, generated by the section $\omega_F := \omega_C \wedge dt$. 

Recall that the translation $t_P$ on $C$ preserves the invariant differential $\omega_C$ (cf.\ \cite{Sil86}*{Chapter III, Proposition 5.1}). 
Since the action of the generator $T \in G$ on $\mathbb{A}^1$ is given by $t \mapsto \xi t$, we have $T^* dt = \xi d t$. 
Therefore, $T$ acts on the section $\omega_F$ by 
\[
T^* \omega_F = T^* \omega_C \wedge T^* dt = \omega_C \wedge (\xi dt) = \xi \omega_F. 
\]
This means that the $G$-linearization on $\mathcal{O}_F(K_F) \cong \mathcal{O}_F \cdot \omega_F$ is given by multiplication by $\xi$. 
Since $\xi$ is a primitive $m$-th root of unity, the $m$-th tensor power $\omega_F^{\otimes m}$ is invariant under $G$, while no smaller positive power is invariant. 
Consequently, the descent of $\mathcal{O}_F(K_F)$ to $F'$ is a line bundle whose $m$-th tensor power is trivial, but no smaller positive power is trivial. 
This means that $\mathcal{O}_{F'}(K_{F'})$ is a torsion element of order exactly $m$ in $\operatorname{Pic}(F')$.
\end{proof}

By the claim and the isomorphism $\mathcal{O}_{F'}(c K_{F'}) \cong \mathcal{O}_{F'}$, $c$ must be a multiple of $m$. 
Since we have already shown that $c$ divides $m$, we conclude that $c = m$. 
This completes the proof of Theorem \ref{thm:counterexample}.
\end{proof}

\end{document}